\documentclass[11pt]{amsart}
\usepackage{latexsym}
\usepackage{rotating}
\usepackage{amsfonts}
\usepackage{amssymb,amscd}
\usepackage{amsmath}
\newtheorem{theorem}{Theorem}[section]
\newtheorem{lem}[theorem]{Lemma}     
\newtheorem{cor}[theorem]{Corollary}
\newtheorem{prop}[theorem]{Proposition}
\theoremstyle{definition}
\newtheorem{definition}[theorem]{Definition}

\theoremstyle{remark}

\numberwithin{equation}{section} \theoremstyle{plain}

\usepackage[margin=3.55 cm]{geometry}

\def\C{\mathbb C}

\def\Z{\mathbb Z}

\newcommand{\secref}[1]{Section~\ref{#1}}
\newcommand{\thmref}[1]{Theorem~\ref{#1}}
\newcommand{\lemref}[1]{Lemma~\ref{#1}}

\numberwithin{equation}{section}

\begin{document}

\title[  Palindromic width of graph of groups]{Palindromic width of 
 graph of groups}

\author[Krishnendu  Gongopadhyay]{Krishnendu Gongopadhyay}

\author[Swathi Krishna]{Swathi Krishna} 
\address{Indian Institute of Science Education and Research (IISER) Mohali, Knowledge City, 
Sector 81,  S.A.S. Nagar, Punjab 140306, India}
\email{krishnendug@gmail.com, krishnendu@iisermohali.ac.in}

\address{Indian Institute of Science Education and Research (IISER) Mohali, Knowledge City, 
Sector 81,  S.A.S. Nagar, Punjab 140306, India}
\email{swathi280491@gmail.com}
\subjclass[2010]{   (Primary) 20F65;   (Secondary) 20E06}
\keywords{ palindromic width, graph of groups, HNN extension, amalgamated free product}
\thanks{Gongopadhyay acknowledges partial support from the grant DST/INT/RUS/RSF/P-19.} 
\date{\today}

\begin{abstract} 
In this paper we answer questions raised by  Bardakov and Gongopadhyay in \emph{Comm Algebra, vol 43, issue 11 (2015), 4809--4824}.  We prove that the palindromic width of HNN extension of a group by proper associated  subgroups is infinite. We also prove that the palindromic width of the amalgamated free product of two groups via a proper subgroup is infinite (except when the amalgamated subgroup has index two in each of the factors). Combining these results it follows that  the palindromic width of the fundamental group of a graph of groups is mostly  infinite. 
\end{abstract}
\maketitle
\section{Introduction} 
Words are basic objects in group theory and they are natural sources to view groups as geometric objects. Using words, one can naturally associate a length to each group element, and the maximum of all such lengths gives the notion of a width. The theory of verbal subgroup, that is subgroup determined by a word (for example, the commutator subgroup), and the verbal width have seen many decisive results in recent time, e.g. see \cite{seg}.  
It is natural to ask for widths given by curious classes of `non-verbal' words. In this paper, we consider the width that comes from one such class, viz. the palindromic words. 

\medskip Let $G$ be a group and let $S$ be a generating set with $S^{-1}=S$. A {\textit{word-palindrome}} or simply, \emph{palindrome} in $G$ is a reduced word in $S$ which reads the same forward and backward. 
Palindromic words arise naturally in the investigation of combinatorics of words and have been studied widely from several point of views,  see \cite{pal3} for a survey.  Palindromes in groups have  also appeared in the context of geometry of automorphisms of free groups, for example, see \cite{gj}, \cite{ful}, \cite{bgs}. Gilman and Keen \cite{gk} applied the geometry of palindromes in a two-generator free group to obtain discreteness conditions for two-generator subgroups in ${\rm SL}(2, \C)$. 

For an element $g \in G$, the {\textit{palindromic length}}, $l_{\mathcal{P}}(g)$ is the minimum number $k$ such that $g$ can be expressed as a product of $k$ palindromes.
Then the {\textit{palindromic width}} of $G$ with respect to $S$ is defined as:
$$pw(G,S) = \sup_{g \in G} l_{\mathcal{P}}(g).$$

Bardakov, Shpilrain and Tolstykh \cite{bst} initiated the investigation of palindromic width and proved that the palindromic width of a non-abelian free group is infinite.   Recently, there have been a series of work that aims to understand the palindromic widths in several other classes of groups including relatively free groups.    Bardakov and Gongopadhyay have proved finiteness of palindromic widths of finitely generated free nilpotent groups and certain solvable groups, see \cite{bg1, bg2, BG-solv}. In \cite{bbg},  finiteness of palindromic width of nilpotent products has been proved.  Palindromic widths of wreath products and Grigorchuk groups have been investigated by Fink \cite{fink1, f2}. Riley and Sale have investigated palindromic widths in certain wreath products and solvable groups \cite{rs} using finitely supported functions from $\Z^r$ to the given group. Fink and Thom \cite{ft} have studied palindromic widths in simple groups and yielded the first examples of groups having finite palindromic widths but infinite commutator widths.  

\medskip The work \cite{bt} has been generalized to free product of groups by Bardakov and Tolstykh in \cite{bt}. It has been proved that a free product of two groups, except $\Z_2 \ast \Z_2$,  has infinite palindromic width. In this paper our aim is to investigate the palindromic widths of some other free constructions of groups.  We investigate the palindromic width for HNN extensions and amalgamated free products of groups.  For HNN extensions we have the following. 
\begin{theorem}\label{hnn}  Let $G$ be a group and let $A$ and $B$ be proper isomorphic subgroups of $G$ and $\phi: A \to B$ be an isomorphism. The HNN extension 
$$ G_{\ast} = \langle G,t \ | \  t^{-1}at=  \phi (a),\  a \in A \rangle$$
of G with associated subgroups $A$ and $B$ has infinite palindromic width with respect to the generating set $G \cup \{t, t^{-1}\}$.
\end{theorem}
For amalgamated free product of groups we prove the following theorem  that extends the work of Bardakov and Tolstykh cited above. 
 \begin{theorem}\label{free}
Let $G=A\ast_{C} B$ be the free product of two groups $A$ and $B$ with amalgamated proper subgroup $C$ and $|A:C| \geq 3$, $|B:C| \geq 2$. Then $pw(G, A \cup B)$ is infinite.
\end{theorem}
 
The above two  theorems answer Question 3 and Question 4 in \cite{BG-solv}, and  also Problem 6  and Problem 7 in \cite{vw}. Since non-solvable Baumslag-Solitar groups are special cases of the HNN extensions, this also answers Question 2 in \cite{BG-solv}.

\medskip As an application of the above two theorems, we determine  the palindromic width for  the fundamental group of a graph of groups. We recall that  a \emph{graph of groups} $(G, Y)$ consists of a non-empty, connected graph $Y$, a group $G_P$ for each $P \in$ vert $Y$ and a group $G_e$ for each $e \in$ edge $Y$, together with monomorphisms $G_e \rightarrow G_{\alpha(e)}$ and $G_e \rightarrow G_{\omega(e)}$, where for each edge $e$, $\alpha(e)$ is the initial vertex and $\omega(e)$ is the final vertex, e.g. \cite{ob}, \cite{serre}.  We assume that $G_e= G_{\bar{e}}$. For each $P \in$ vert $Y$, let $S_P$ be the generating set of $G_P$. Also, let $T$ be a maximal tree in $Y$.  We fix
\[S= \{\cup_{P \in \text{vert}\hspace{1mm} Y} S_P\} \cup \{edge(Y)-edge(T)\}\] to be the \emph{standard generating set} of the fundamental group of $(G, Y)$, $\pi_1(G, Y)$. It is straight forward to see that  the fundamental group of any graph of groups has a representation which is an amalgamated free product or a HNN extension. Hence we have the following consequence of \thmref{hnn} and \thmref{free}. 
\begin{cor}\label{gog} 
Let $Y$ be a non-empty, connected graph. Let $\pi_1(G, Y)$ be the fundamental group of the graph of groups of $Y$ with the standard generating set $S$. Then the palindromic width of $\pi_1(G, Y)$ is infinite if
\begin{enumerate}
\item $Y$ is a loop with a vertex $P$ and edge $e$; and the image of $G_e$ is a proper subgroup of $G_P$; or
\item $Y$ is a tree and has an oriented edge $e= [P_1, P_2]$  such that removing $e$, while retaining $P_1$ and $P_2$, gives two disjoint graphs $Y_1$ and $Y_2$ with $P_i \in$ vert $Y_i$ satisfying the following: extending $G_e \rightarrow G_{P_i}$ to $\phi_i : G_e \rightarrow \pi_1(G, Y_i)$, $i= 1,2$, we get $[\pi_1(G, Y_1) : \phi_1(G_e)] \geq 3$ and $[\pi_1(G, Y_2) : \phi_2(G_e)] \geq 2$.
\item $Y$ has an oriented edge $e = [P_1,P_2]$ such that removing the edge, while retaining $P_1$ and $P_2$ does not separate $Y$ and gives a new graph $Y'$ satisfying the following:
 extending $G_e \rightarrow G_{P_i}$ to $\phi_i : G_e \rightarrow \pi_1(G, Y')$, $i= 1,2$, we have $\phi_{i}(G_e) = H_{i}$ and $H_1$, $H_2$ are proper subgroups of $\pi_1(G, Y')$.
\end{enumerate} 
\end{cor}

 The fundamental group in $(1)$ is an HNN extension of $G_P$ and so, $(1)$ follows from \thmref{hnn}. In $(2)$, the fundamental group is an amalgamated free product of $\pi_1(G, Y_1)$ and $\pi_1(G, Y_2)$ with proper amalgamated subgroups $\phi_1(G_e) \cong \phi_2(G_e)$. The result follows from \thmref{free}. Finally, the fundamental group in $(3)$ is an HNN extension of $G'$, with $G'$ being the fundamental group of the graph of groups corresponding to $Y'$. Hence, this also follows from \thmref{hnn}. 
 
 \subsubsection*{Idea of the Proof} Let $G$ be the group under consideration. An element $g$ in $G$ is a \emph{group-palindrome}  if $g$ can be represented by a word $w$ such that its reverse $\bar{w}$ also represents $g$. This notion is  weaker than the notion of `word-palindromes', see \cite{bbg} for a comparison of these two notions.  The set $\mathcal P$ of word-palindromes is obviously a subset  of $\mathcal{ GP}$, the set of group-palindromic words.  Thus, for an element $g$ in $G$, $ l_{\mathcal{GP}}(g) \leq l_{\mathcal P}(g)$. 
Consequently, the palindromic width with respect to group-palindromes does not exceed $pw(G, S)$. We shall show that the palindromic width with respect to group-palindromes is infinite and that will establish the main results. To achieve this, we shall use quasi-morphism techniques. 
\begin{definition}
Let $H$ be a group. A map $\Delta: H \to \mathbb{R} $ is called a {\textit{quasi-homomorphism}} if there exists a constant $c$ such that for every $x, y \in H$, $\Delta(xy) \leq \Delta(x) + \Delta(y) +c$.
\end{definition} 
Quasi-morphisms have wide applications in mathematics and for a brief survey on these objects see \cite{kot}. However, our motivation for using them in this paper comes from the work of Bardakov \cite{vb} and Dobrynina \cite{dob1, dob2} where the authors have proved infiniteness of verbal subgroups of HNN-extensions and amalgamated free products, also see \cite{bst}, \cite{bt}. We shall follow their methods here.  

We prove \thmref{hnn} in \secref{hnns} and  \thmref{free} in \secref{afp}.  For both the proofs,  a canonical form for a  palindromic word will be obtained. Then it will be shown that the underlying quasi-homomorphism $\Delta$  will be bounded if $G$ has finite  palindromic width.  Finally,   a sequence of elements in the underlying group will be noted where $\Delta$ will be bounded away, thus establishing the infiniteness of the palindromic widths.  and proving the theorems.

\bigskip \noindent{\bf Notation.} 
Let $f$ and $g$ be functions over non-zero integers. We write $f =_m g$ to denote that $f(k)=g(k)$ for all values of $k$ except at most $m$ values. Then clearly, $f =_m g$ and $g =_n h$ implies $f =_{m+n} h$. Also $f =_m g$ and $f' =_n g'$ implies $f+f' =_{m+n} g+g'$.

\section{Palindromic Width for  HNN Extensions of Groups}\label{hnns}
\subsection{HNN Extensions} 
Let $G$ be a group and $A$ and $B$ be proper isomorphic subgroups of $G$ with the isomorphism $\phi: A\to B$. Then the HNN extension of $G$ is 
$$G_{\ast} = \langle G,t~|~t^{-1}at=\phi(a), a\in A\rangle .$$ 

\medskip A sequence $g_{0},t^{\epsilon_{1}},g_{1},t^{\epsilon_{2}}, \ldots ,g_{n-1},t^{\epsilon_{n}},g_{n}$, $n\geq 0$, is said to be{\textit{ reduced}} if it does not contain subsequences of the form $t^{-1},g_{i},t$ with $g_{i} \in$ A or $t,g_{i},t^{-1}$ with $g_{i} \in$ B. 
By Britton's Lemma, if a sequence $g_{0},t^{\epsilon_{1}},g_{1},t^{\epsilon_{2}}, \ldots ,g_{n-1},t^{\epsilon_{n}},g_{n}$ is reduced and $n \geq  1$, then $g = g_{0}t^{\epsilon_{1}}g_{1}t^{\epsilon_{2}} \ldots g_{n-1}t^{\epsilon_{n}}g_{n}$ is not trivial in $G_{\ast}$ and we call it a {\textit{reduced word}}.

\vspace{3mm}

Such a representation of a group element of an HNN extension is not unique but the following lemma holds:

\begin{lem}\label{lem1}\cite[Lemma 3]{vb}
Let $g= g_{0}t^{\epsilon_{1}}g_{1}t^{\epsilon_{2}} \ldots g_{n-1}t^{\epsilon_{n}}g_{n}$ and \\ $h= h_{0}t^{\theta_{1}}h_{1}t^{\theta_{2}} \ldots h_{m-1}t^{\theta_{m}}h_{m}$ be reduced words, and suppose $g= h$ in $G_{\ast}$. Then $m= n$ and $\epsilon_i= \theta_i$ for $i= 1, \ldots ,n$. 
\end{lem}

\begin{definition} 
The {\textit{signature}} of $g\in G_{\ast}$ is the sequence $sqn(g)=(\epsilon_{1},\epsilon_{2}, \ldots ,\epsilon_{n})$, $\epsilon_{i} \in \{1,-1\}$ for $g = g_{0}t^{\epsilon_{1}}g_{1}t^{\epsilon_{2}} \ldots g_{n-1}t^{\epsilon_{n}}g_{n}$.
\end{definition}

By \lemref{lem1}, the signature of any $g \in G_{\ast}$ is unique, irrespective of the choice of the reduced word.

\vspace{3mm}

Let $\sigma$=($\epsilon_{1},\epsilon_{2}, \ldots ,\epsilon_{n}$) be a signature. Then the length of the signature, $|\sigma|=n$. And the inverse signature, $\sigma^{-1}$=($-\epsilon_{n},-\epsilon_{n-1}, \ldots ,-\epsilon_{1}$). So, $sqn(g^{-1}) =(sqn (g))^{-1}$.

\vspace{3mm}

Product of two signatures $\sigma$ and $\tau$, $\sigma\tau$, is obtained by writing $\tau$ after $\sigma$.

\vspace{3mm}

Suppose $\sigma= \sigma_{1}\rho$ and $\tau=\rho^{-1}\tau_{1}$ with $|\rho|=r$, then we can define an $r$-product,
$$\sigma[r]\tau = \sigma_{1}\tau_{1}.$$

The following lemma is immediate from the above notions. 
\begin{lem}\cite[Lemma 4]{vb}
For any $g,h \in G_{\ast}$, there exists an integer r $\geq$ 0 such that $sqn(gh)= sqn(g)[r]sqn(h)$, with $sqn(g)= \sigma_{1}\rho$ and $sqn(h)=\rho^{-1}\tau_{1}$ and $|\rho|=r$.
\end{lem}

A reduced expression is called {\textit{positive (negative)}} if all exponents $\epsilon_{i}$ are positive (resp. negative). Further, if it is either positive or negative then the reduced expression is called {\textit{homogeneous}}.

\subsection{Proof of \thmref{hnn}}  \;\

\medskip Let $\sigma$=($\epsilon_{1},\epsilon_{2}, \ldots ,\epsilon_{n}$) be the signature of an $g \in G_{\ast}$.  We define, 

$p_{k}(g)$ = number of $+1, +1,  \ldots ,+1$  sections of length $k$, 

$m_{k}(g)$ = number of $ -1,-1, \ldots ,-1$  sections of length $k$,

$d_{k}(g)$= $p_{k}(g)-m_{k}(g)$, 

$r_{k}(g)$= remainder of $d_k(g)$ divided by 2, and, 
$$\Delta(g) = \sum_{k=1}^{\infty} r_{k}(g).$$

\medskip Clearly, $p_{k}(g^{-1})= m_{k}(g)$ and so, $d_{k}(g^{-1}) +d_{k}(g)=0$ for all $g\in G_{\ast}$.

\begin{lem}
For any elements $g, h \in G_{\ast}$,  $\Delta(gh) \leq \Delta(g)+ \Delta(h)+6$, i.e. $\Delta$ is a quasi-homomorphism. 
\end{lem}

\begin{proof}
The proof follows from \cite[Lemma 9]{vb}.
\end{proof}

\begin{definition}
 Let  $g= g_{0}t^{\epsilon_{1}}g_{1}t^{\epsilon_{2}} \ldots g_{n-1}t^{\epsilon_{n-1}}g_{n}$ be a reduced element in $G_{\ast}$. Put \[\bar{g}= g_{n}t^{\epsilon_{n-1}} g_{n-1}t^{\epsilon_{n-2}} \ldots g_{1}t^{\epsilon_{1}} g_{0}.\]  We say $g$ is a group-palindrome if $\bar{g}= g$ and $\bar{g}$ depends on the reduced form.
\end{definition}

\begin{lem}
 A group-palindrome $g \in G_{\ast}$ has the form 
\[g= 
\begin{cases}
g_{0}t^{\epsilon_{1}}g_{1} \ldots g_{k-1}t^{\epsilon_k}g_k't^{\epsilon_k}{g_{k-1}} \ldots {g_{1}}t^{\epsilon_{1}}{g_{0}}, & \text{if }\   |sqn(g)|=2k, \\
g_{0}t^{\epsilon_{1}}g_{1} \ldots t^{\epsilon_k}g_kt^{\epsilon_{k+1}}{g_k'}t^{\epsilon_k} \ldots {g_{1}}t^{\epsilon_{1}}{g_0}, & \text{if }\    |sqn(g)|=2k+1, \\
\end{cases}
\] where $g_k'= xg_k$ where $x \in A \cup B$.
\end{lem}

\begin{proof} Let $g \in G_{\ast}$ is a group-palindrome. 

CASE 1: $|sqn(g)|=2k+1$.

Let $g = g_{0}t^{\epsilon_{1}}g_{1}t^{\epsilon_{2}} \ldots g_{2k}t^{\epsilon_{2k+1}}g_{2k+1}$.
We know, $g = \bar{g}$.
\[\Rightarrow g{\bar{g}}^{-1} = 1\]
\[\Rightarrow g_{0}t^{\epsilon_{1}}g_{1} \ldots t^{\epsilon_{2k+1}}g_{2k+1}g_{0}^{-1}t^{-\epsilon_{1}}g_{1}^{-1} \ldots t^{-\epsilon_{2k+1}}g_{2k+1}^{-1} = 1\]

The left side is reducible. So we have,
$g_{2k+1}g_{0}^{-1} = x_0$, where $x_0 \in A$ (or $x_0 \in B$) such that $t^{\epsilon_{2k+1}}x_0t^{-\epsilon_{1}} = y_0$, with 
$y_0 \in B$ (or $y_0 \in A$) and $\epsilon_{2k+1} = \epsilon_{1}= -1$ (or $1$).

\[\Rightarrow g_{0}t^{\epsilon_{1}}g_{1} \ldots t^{\epsilon_{2k}}g_{2k}y_0g_{1}^{-1}t^{-\epsilon_{2}} \ldots g_{2k+1}^{-1} = 1\]

Since $y_0 \in B$ (or $y_0 \in A$),  $g_{2k}{y_0}{g_1}^{-1} = y_1$, $y_1 \in B$ (or $y_1 \in A$) such that $ t^{\epsilon_{2k}}y_1t^{-\epsilon_{2}} = x_1$, where $x_1 \in A$ (or $x_1 \in B$) and $\epsilon_{2k} = \epsilon_{2} = 1$ (or $-1$). 

\[\Rightarrow g_{0}t^{\epsilon_{1}}g_{1} \ldots t^{\epsilon_{2k-1}}g_{2k-1}x_1g_{2}^{-1}t^{-\epsilon_{2}} \ldots g_{2k+1}^{-1} = 1\]

Since $x_1 \in A$ (or $x_1 \in B$),  $g_{2k-1}x_1{g_2}^{-1} = x_2$, $x_2 \in A$ (or $x_2 \in B$) such that $ t^{\epsilon_{2k-1}}x_2t^{-\epsilon_{3}} = y_2$, where $y_2 \in B$ (or $y_2 \in A$) and $\epsilon_{2k-1} = \epsilon_{2} = -1$ (or $1$).

\medskip In general, we get $g_{2k-i}x_{i}g_{i+1}^{-1} = x_{i+1}, $ 
$x_{i}, x_{i+1} \in A$ (or $B$)
 such that
  $t^{\epsilon_{2k-i}}x_{i+1}t^{-\epsilon_{i+2}} = y_{i+1}$, where $y_{i+1} \in B$ (or $A$) 
 and $\epsilon_{2k-i} = \epsilon_{i+2}$, where $0 \leq i \leq k-1$.

\medskip In the expression 
$g = g_{0}t^{\epsilon_{1}}g_{1} t^{\epsilon_{2}}\ldots g_{k}t^{\epsilon_{k+1}}g_{k+1}t^{\epsilon_{k+2}}g_{k+2} \ldots t^{\epsilon_{2k+1}}g_{2k+1}$,

we put $g_{2k+1} = x_0g_0$, and for $0 \leq i \leq k-1$, $g_{2k-i} = x_{i+1}g_{i+1}x_{i}^{-1}$ and $\epsilon_{2k-i} = \epsilon_{i+2}$.
\[\Rightarrow g = g_{0} \ldots g_{k}t^{\epsilon_{k+1}} x_{k}g_kx_{k-1}^{-1}t^{\epsilon_{k}} y_{k-1}g_{k-1}y_{k-2}^{-1}t^{\epsilon_{k-1}}  \ldots y_{1}g_{1}y_{0}^{-1}t^{\epsilon_1}{x_0}{g_0}\]

We know $t^{\epsilon_{i+2}}x_{i+1} = y_{i+1}t^{\epsilon_{i+2}}$ (or $t^{\epsilon_{i+2}}y_{i+1} = x_{i+1}t^{\epsilon_{i+2}}$) for $-1 \leq i \leq k-2$;

\[\Rightarrow g = g_{0} \ldots g_{k}t^{\epsilon_{k+1}} x_{k}g_kx_{k-1}^{-1}x_{k-1}t^{\epsilon_{k}} g_{k-1}y_{k-2}^{-1}y_{k-2}t^{\epsilon_{k-1}}  \ldots x_{1}t^{\epsilon_2}g_{1}y_{0}^{-1}{y_0}t^{\epsilon_1}{g_0}\]
\[\Rightarrow g = g_{0} \ldots g_{k}t^{\epsilon_{k+1}} x_{k}g_kt^{\epsilon_{k}} g_{k-1}t^{\epsilon_{k-1}}  \ldots t^{\epsilon_2}g_{1}t^{\epsilon_1}{g_0}\]

Therefore, \[g = g_{0} \ldots g_{k}t^{\epsilon_{k+1}} {g'_k}t^{\epsilon_{k}}g_{k-1}t^{\epsilon_{k-1}}g_{k-2} \ldots g_{1}t^{\epsilon_1}{g_0};\] where ${g'_k} = x_{k}g_k$.

\vspace{5mm}

CASE 2: $|sqn(g)|=2k$.

Let $g = g_{0}t^{\epsilon_{1}}g_{1}t^{\epsilon_{2}} \ldots g_{2k-1}t^{\epsilon_{2k}}g_{2k}$. 
We know, $g = \bar{g}$. This implies, \[g_{0}t^{\epsilon_{1}}g_{1} \ldots t^{\epsilon_{2k}}g_{2k}g_{0}^{-1}t^{-\epsilon_{1}}g_{1}^{-1} \ldots t^{-\epsilon_{2k}}g_{2k}^{-1} = 1\]

The left side is reducible. So we have, $g_{2k}g_{0}^{-1} = x_0$, where $x_0 \in A$ (or $x_0 \in B$) such that $t^{\epsilon_{2k}}x_0t^{-\epsilon_{1}} = y_0$, with 
$y_0 \in B$ (or $y_0 \in A$) and $\epsilon_{2k} = \epsilon_{1} = -1$ (or $1$).
 
\[\Rightarrow g_{0}t^{\epsilon_{1}}g_{1} \ldots t^{\epsilon_{2k-1}}g_{2k-1}y_0g_{1}^{-1}t^{-\epsilon_{2}} \ldots t^{-\epsilon_{2k}}g_{2k}^{-1} = 1.\]

Since $y_0 \in B$ (or $y_0 \in A$), $g_{2k-1}{y_0}{g_1}^{-1} = y_1$, $y_1 \in B$ (or $y_1 \in A$) such that $ t^{\epsilon_{2k-1}}y_1t^{-\epsilon_{2}} = x_1$, where $x_1 \in A$ (or $x_1 \in B$) and $\epsilon_{2k-1} = \epsilon_{2}$.

\[\Rightarrow g_{0}t^{\epsilon_{1}}g_{1} \ldots t^{\epsilon_{2k-2}}g_{2k-2}x_1g_{2}^{-1}t^{-\epsilon_{3}} \ldots t^{-\epsilon_{2k}}g_{2k}^{-1} = 1.\]
 
Similarly, since $x_1 \in A$ (or $x_1 \in B$), $g_{2k-2}{x_1}{g_2}^{-1} = x_2$, $x_2 \in A$ (or $x_2 \in B$) such that $ t^{\epsilon_{2k-2}}x_2t^{-\epsilon_{3}} = y_2$, where $y_2 \in B$ (or $y_2 \in A$) and $\epsilon_{2k-2} = \epsilon_{3}$.

\[\Rightarrow g_{0}t^{\epsilon_{1}}g_{1} \ldots t^{\epsilon_{2k-3}}g_{2k-3}y_2g_{3}^{-1}t^{-\epsilon_{4}} \ldots t^{-\epsilon_{2k}}g_{2k}^{-1} = 1.\]

\medskip In general, we get $g_{2k-i}{x_{i-1}}g_{i}^{-1} = x_{i}$, $x_{i-1}, x_{i} \in A$ (or $B$) such that $ t^{\epsilon_{2k-i}}x_{i}t^{-\epsilon_{i+1}} = y_{i}$, where $y_{i} \in B$ (or $A$) and $\epsilon_{2k-i} = \epsilon_{i+1}$, where $1 \leq i \leq k$.

\medskip In $g = g_{0}t^{\epsilon_{1}}g_{1} \ldots t^{\epsilon_{2k}}g_{2k}$, 

we put $g_{2k} = x_0g_0$ and for $1 \leq i \leq k$, 
$g_{2k-i} = x_ig_ix_{i-1}^{-1}$ and $\epsilon_{2k-i} = \epsilon_{i+1}$.

\[\Rightarrow g = g_{0}t^{\epsilon_{1}}g_{1} \ldots t^{\epsilon_{k}}x_kg_kx_{k-1}^{-1}t^{\epsilon_{k}}y_{k-1}g_{k-1}y_{k-2}^{-1}t^{\epsilon_{k-1}} \ldots t^{\epsilon_{3}}x_{2}g_{2}x_{1}^{-1}t^{\epsilon_{2}}y_{1}g_{2}y_{0}^{-1}t^{\epsilon_{1}}x_0g_0 \]

Put $t^{\epsilon_{i+1}}x_i = y_it^{\epsilon_{i+1}}$ for $-1 \leq i \leq k-1$.

\[\Rightarrow g = g_{0}t^{\epsilon_{1}}g_{1} \ldots t^{\epsilon_{k}}x_kg_kx_{k-1}^{-1}x_{k-1}t^{\epsilon_{k}}g_{k-1}y_{k-2}^{-1}y_{k-2}t^{\epsilon_{k-1}} \ldots y_{2}t^{\epsilon_{3}}g_{2}x_{1}^{-1}x_{1}t^{\epsilon_{2}}g_{2}y_{0}^{-1}{y_0}t^{\epsilon_{1}}g_0 \]

\[\Rightarrow g = g_{0}t^{\epsilon_{1}}g_{1} \ldots t^{\epsilon_{k}}x_kg_kt^{\epsilon_{k}}g_{k-1}t^{\epsilon_{k-1}} \ldots t^{\epsilon_{3}}g_{2}t^{\epsilon_{2}}g_{2}t^{\epsilon_{1}}g_0 \]

Therefore, \[g = g_{0} \ldots {g'_k}t^{\epsilon_{k-1}}g_{k-1} \ldots  g_1t^{\epsilon_1}g_0,\] where ${g'_k} = x_{k}g_{k}$.
\end{proof}

\begin{lem}
Let $g\in G_{\ast}$ be a product of $k$ group-palindromes, say $g= p_1p_2 \ldots p_k$. Then, $\Delta(g) \leq 7k-6$.\end{lem}

\begin{proof}
 Let $p$ be a  group-palindrome in $G_{\ast}$ of non-zero length.

Then $p$ can be represented as $p=uv\bar u$, where $v$ is the maximal homogeneous palindromic sub-word in $p$ and $\bar u$ is $u$ written in reverse.

For example, if $p = g_{0}t^{\epsilon_{1}}g_{1} \ldots t^{-1}g_{i} t g_{i+1} t \ldots
 t  \bar g_{i+1} t \bar g_{i} t^{-1} \bar g_{i-1} \ldots \bar g_{1} t^{\epsilon_{1}} \bar g_{0}$, then 
\[u=g_{0}t^{\epsilon_{1}}g_{1} \ldots t^{-1}, \]
\[v=g_{i}tg_{i+1}t \ldots t\bar g_{i+1}t\bar g_{i}, \]
\[\bar u=t^{-1}\bar g_{i-1} \ldots  \bar g_{1}t^{\epsilon_{1}} \bar g_{0}. \]

Then for every $k$, $d_k(u)=d_k(\bar u)$. As $v$ is homogeneous, if $k'$ is the length of $sqn(v)$, then 
$$p_{k'}(p)= 2p_{k'}(u)+ p_{k'}(v), \hbox{  or } m_{k'}(p)= 2m_{k'}(u)+ m_{k'}(v).$$
For all other $k$,  $p_{k}(p)= 2p_{k}(u)$ and $m_{k}(p)= 2m_{k}(u)$.

Therefore,  $$ r_{k'}(p)= 1, 
 \hbox{ and } r_{k}(p)= 0 \hbox{ for all other  } k.$$

Thus, \[ \Delta(p) =1.\] If $p\in G$, then $\Delta(p)=0$. So, $\Delta(p) \leq 1$.

\medskip Then, if $g\in G_{\ast}$ is a product of $k$ group-palindromes, say $g= p_1p_2 \ldots p_k$, then 
\begin{equation}\label{1}
\Delta(g)= \Delta(p_1p_2 \ldots p_k) \leq \Delta(p_1)+\Delta(p_2)+ \cdots +\Delta(p_k)+6(k-1) \leq 7k-6.
\end{equation}
This completes the proof. 
\end{proof}
\subsection{Proof of \thmref{hnn}}
Now we prove that $\Delta$ is not bounded from above. For that purpose, we produce the following sequence of reduced words $\{a_i\}$, for which $\Delta(a_i)$ is increasing.
\[ \hbox{Let } a_1 = g_0tg_1t^{-1}g_2tg_3.\]

Then $d_1(a_1)= 1$, so $\Delta(a_1)=1$.
\[\hbox{For } a_2 = g_0tg_1t^{-1}g_2tg_3t^{-1}g_4t^{-1}g_5tg_6tg_7t^{-1}g_8t^{-1}g_9.\]

$d_1(a_2)= 1, ~d_2(a_2)=-1$,  so, $\Delta(a_2)=2$.
\[a_3= g_0tg_1t^{-1}g_2tg_3t^{-1}g_4t^{-1}g_5tg_6tg_7t^{-1}
g_8t^{-1}g_9tg_{10}tg_{11}tg_{12}t^{-1}g_{13}
t^{-1}g_{14}t^{-1}g_{15}tg_{16}t
g_{17}tg_{18}.\] 

Then, 
$d_1(a_3)=1, d_2(a_3)=-1, d_3(a_3)=1$, so,  $\Delta(a_3)=3$.

For each $a_i = g_0tg_1t^{-1}g_2t\ldots$, we have $g_j\in G$ and since $a_i$ is reduced, for subwords of the form $t^{\epsilon}g_{i}t^{-\epsilon}$, $g_i \notin A$ if $\epsilon = -1$ and $g_i \notin B$ if $\epsilon = 1$. 

Given $a_i$, we construct $a_{i+1}$ by attaching a segment with signature of length $3(i+1)$. In  general, 
\[a_n = g_0tg_1t^{-1}g_2tg_3...g_{N-2n}t^{\mp 1}...g_{N-n-1}t^{\mp 1}g_{N-n}t^{\pm 1}g_{N-n+1}t^{\pm 1}...g_{N-1}t^{\pm 1}g_{N};
\] where $N = \frac{3n(n+1)}{2}$; and
 \[sqn(a_n) = (1,-1,1,-1,-1,1,1,-1,-1, \ldots ,\underbrace{\mp 1, \ldots ,\mp 1}_\text{$n-1$ times},\underbrace{\pm 1, \ldots ,\pm 1}_\text{$n$ times},\underbrace{\mp 1, \ldots ,\mp 1}_\text{$n$ times},\underbrace{\pm 1, \ldots ,\pm 1}_\text{$n$ times})\]

\[\Delta(a_n)=n.\]
Then, by  \eqref{1}, we get that the palindromic width of $G_{\ast}$ is infinite. This proves \thmref{hnn}. 

\section{Palindromic Width for Amalgamated Free Products}\label{afp}
First recall the notion of free product with amalgamation.  Let $A= \langle a_1, \ldots | R_1, \ldots \rangle $ and $B= \langle b_1, \ldots |S_1, \ldots \rangle $ be groups. Let $C_1 \subset A$ and $C_2 \subset B$ be subgroups such that there exists an isomorphism $\phi: C_1 \to C_2$. Then the {\textit{free product of $A$ and $B$}}, amalgamating the subgroups $C_1$ and $C_2$ by the isomorphism $\phi$ is the group
$$G=\langle A,B\, |\, c=\phi(c), c \in C_1\rangle.$$
We can view $G$ as the quotient of the free product $A \ast B$ by the normal subgroup generated by $\{c \phi(c)^{-1}| c \in C_1\}$. The subgroups $A$ and $B$ are called factors of $G$, and since $C_1$ and $C_2$ are identified in $G$, we will denote them both by $C$.

\medskip We shall divide the proof of \thmref{free} into two cases. 

\subsection{Case 1} \label{ca1} For a non-trivial $a \in A\cup B$ such that $CaC \neq Ca^{-1}C$. We shall prove the following:

\begin{lem}\label{fl2} 
Let $G=A\ast_{C} B$ be the free product of two groups $A$ and $B$ with amalgamated subgroup $C$. Let $|A:C| \geq 3$, $|B:C| \geq 2$ and there exists an element $a \in A\cup B$ for which $CaC \neq Ca^{-1}C$.  Then $pw(G, \{A \cup B\})$ is infinite.
\end{lem}

\medskip To prove this, we shall use the quasi-homomorphism constructed in \cite{dob1, dob2}. We recall the construction here. 

\subsubsection{Quasi-homomorphisms} 
\begin{definition}
A sequence $x_{1}, \ldots ,x_{n}$, $n \geq 0$, is said to be{\textit{ reduced}} if 
\begin{enumerate}
\item Each $x_i$ is in one of the factors.
\item Successive $x_i$, $x_{i+1}$ come from different factors.
\item If $n>  1$, no $x_i$ is in $C$.
\item If $n=1$, $x_1 \neq 1$.
\end{enumerate}
\end{definition}

By the normal form of elements in free products with amalgamation, see for eg.  \cite{ls}, if $x_{1}, \ldots,x_{n}$  is a reduced sequence, $n \geq 1$, then the product $x_{1} \ldots x_{n} \neq 1$ is in $G$ and it is called a reduced word. Such a representation of a group element is not unique but the following lemma holds:

\begin{prop}
Let $g= x_{1} \ldots x_{n}$ and $h= y_{1} \ldots y_{m}$ be reduced words such that $g= h$ in $G$. Then $m= n$.
\end{prop}

\begin{proof}
Since $g= h$, we have,
\[1= x_{1} \ldots x_{n}y_{m}^{-1} \ldots y_{2}^{-1}y_{1}^{-1}.\]
Since $g$ and $h$ are reduced, we require $x_{n}y_{m}^{-1}$ to belong to $C$. To reduce it further we need $x_{n-1}x_{n}y_{m}^{-1}y_{m-1}^{-1}$ to be in $C$ and so on. Hence, $m= n$.
\end{proof}

\begin{definition}
	Let $g= x_1 \ldots x_n$ be a reduced word of $g \in G$. The elements $x_k$ are said to be {\textit{syllables}} of $g$. 	
	Then the \textit{length} of $g$ is the number of syllables of $g$ and it is denoted by $l(g)$. Here, for $g= x_1 \ldots x_n$, $l(g) = n$.
\end{definition}	

%

\begin{definition}
	Let $a \in A$ such that $CaC \neq Ca^{-1}C$. Let $g \in G$, and $g= x_1x_2 \ldots x_n$ be a reduced word representing it. Then the {\textit{special form}} of $g$ associated to this reduced word is obtained by replacing $x_i$ by $ua^{\epsilon}u'$, whenever $x_i = ua^{\epsilon}u'$ for some $u,u' \in C$ and $\epsilon \in \{+1,-1\}$, in the following way: 
	\begin{enumerate}
		\item When $i=1$, $x_1 = ua^{\epsilon}u'$, we write $g = ua^{\epsilon}x_2' \ldots x_n$, where $x_{2}'= u'x_2$.
		\item When $2 \leq i \leq n-1$, $x_i = ua^{\epsilon}u'$, we write $g = x_1x_2 \ldots x_{i-1}'a^{\epsilon}x_{i+1}' \ldots x_n$, where $x_{i-1}'= x_{i-1}u$ and $x_{i+1}' = u'x_{i+1}$.
		\item When $i=n$ and $x_n = ua^{\epsilon}u'$, where $\epsilon \in \{+1, -1\}$ and $u,u' \in C$, we replace $x_n$ by  $g= x_1x_2 \ldots x_{n-1}'a^{\epsilon}u'$, where $x_{n-1}'= x_{n-1}u$.
	\end{enumerate}
\end{definition}

An \emph{$a$-segment} of length $2k-1$  is a segment of the reduced word of the following form
\[ax_1 \ldots x_{2k-1}a,\]
where $x_j \neq a$ for $j= 1, \ldots, 2k-1$ such that the length of $x_1 \ldots x_{2k-1}$ is $2k-1$.

Similarly, an \emph{$a^{-1}$-segment} of length $2k-1$  is a segment of the reduced word of the following form
\[a^{-1}x_1 \ldots x_{2k-1}a^{-1},\]
where $x_j \neq a^{-1}$ for $j= 1, \ldots, 2k-1$ such that the length of $x_1 \ldots x_{2k-1}$ is $2k-1$.

\vspace{2mm}
For $g \in G$ expressed in special form, we define 

$p_{k}(g)$ = number of $a$-segments of length $2k-1$, 

$m_{k}(g)$ = number of $a^{-1}$-segments of length $2k-1$,

$d_{k}(g)$= $p_{k}(g)-m_{k}(g)$,

$r_{k}(g)$= remainder of $d_k(g)$ divided by 2,  and  

\begin{equation} \label{eq1}
\Delta(g) = \sum_{k=1}^{\infty} r_{k}(g)
\end{equation}

Clearly, $p_{k}(g^{-1})= m_{k}(g)$ and so, $d_{k}(g^{-1}) +d_{k}(g)=0$ for all $g\in$ G.

\begin{lem}
	$\Delta$ is well-defined on special forms.
\end{lem}
\begin{proof}
	
	Let $g \in G$ and $x_1x_2\ldots x_n$ and $y_1y_2\ldots y_n$ be two reduced forms of $g$. 
	
	Now, $x_1x_2\ldots x_n = y_1y_2\ldots y_n$ implies $x_1\ldots x_ny_{n}^{-1}\ldots y_{1}^{-1} = 1$. Then $x_ny_{n}^{-1} = c_n \in C$. Further $x_{n-1}c_{n}y_{n-1}^{-1} = c_{n-1} \in C$ and so on. In general, for $1 \leq i \leq n$, $x_{i}c_{i+1}y_{i}^{-1} = c_{i} \in C$.
	
	So, if $x_i = ua^{\epsilon}u'$ for some $u, u' \in C$ and $\epsilon \in \{+1,-1\}$, $ua^{\epsilon}u'c_{i+1}y_{i}^{-1} = c_{i}$. This gives $va^{\epsilon}v' = y_{i}$, where $v= c_{i}^{-1}u \in C$ and $v' = u'c_{i+1} \in C$. 
	
	Thus, for any $k \in \mathbb{N}$, number of $a^{\epsilon}$ segments of length $2k-1$, for $\epsilon \in \{+1,-1\}$ is independent of the special form of $g \in G$. Thus, $\Delta$ is well-defined on special forms. 	
	\end{proof}

\begin{lem}
For any elements $g,h \in G$, $\Delta(gh) \leq \Delta(g)+ \Delta(h)+9$, i.e. $\Delta$ is a quasi-homomorphism. 
\end{lem}
\begin{proof}
See the proof of \cite[Lemma 2]{dob1}.
\end{proof}

\subsubsection{Normal form of palindromes}
\begin{definition}
Let $g= x_1 \ldots x_n$ be a reduced word of $g \in G$. Let $\bar{g}$ be the word obtained by writing $g$ in the reverse order, i.e. $\bar{g}= {x}_n \ldots {x}_1$. This is a non-trivial element of $G$. We say $g$ is a {group-palindrome} if  $\bar{g}= g$. 
\end{definition}

\begin{lem}\label{lem2}
A group-palindrome  $g \in G$ has the form 
\[g= 
x_1x_2 \ldots x_kx_{k+1}'{x_k}{x_{k-1}} \ldots {x_1}\] where $x_{k+1}'= x_{k+1}c$ with $c\in C$.
\end{lem}

\begin{proof} Let $g$ is a group-palindrome in $G$. 

CASE 1: $l(g)=2k+1$.

Let \[g= x_1x_2 \ldots {x_k}{x_{k+1}} \ldots x_{2k}x_{2k+1}.\]
We know $g = {\bar{g}}$
\[x_1x_2 \ldots {x_k}{x_{k+1}} \ldots x_{2k}x_{2k+1} = x_{2k+1}x_{2k} \ldots x_2x_1\]
\[\Rightarrow x_1x_2 \ldots x_{2k}x_{2k+1}{x_1}^{-1}{x_2}^{-1} \ldots {x_{2k}}^{-1}{x_{2k+1}}^{-1} = 1 \]

Since the expression on the left side is reducible,  $x_{2k+1}{x_1}^{-1} = c_1$; for $c_1 \in C$. This implies, 
$x_{2k+1} = {c_1}x_1$. Thus, $$x_1x_2 \ldots x_{2k}{c_1}{x_2}^{-1} \ldots {x_{2k-1}}^{-1}{x_{2k}}^{-1} = 1.$$

Further $x_{2k}{c_1}{x_2}^{-1} = c_2$; for $c_2 \in C$ $\Rightarrow x_{2k} = c_2{x_2}c_{1}^{-1}$.
\[\Rightarrow x_1x_2 \ldots x_{2k-1}{c_2}{x_3}^{-1} \ldots {x_{2k-1}}^{-1}{x_{2k}}^{-1} = 1 \]

In general we get $x_{2k-i} = c_{i+2}{x_{i+2}}c_{i+1}^{-1}$ for $0 \leq i \leq k-2$.

Then, \[g= x_1x_2 \ldots {x_k}{x_{k+1}} \ldots x_{2k-1}x_{2k}.\]
\[\Rightarrow g = x_1x_2 \ldots {x_k}x_{k+1}c_{k}{x_k}c_{k-1}^{-1}c_{k-1}{x_{k-1}}c_{k-2}^{-1} \ldots c_{3}{x_{3}}c_{2}^{-1}c_{2}{x_{2}}c_{1}^{-1}c_{1}{x_{1}}\]
\[\Rightarrow g = x_1x_2 \ldots {x_k}x_{k+1}c_{k}x_{k} \ldots {x_{3}}{x_{2}}{x_{1}}.\]

Therefore, $ g = x_1x_2 \ldots x'_{k+1}x_{k} \ldots {x_{3}}{x_{2}}{x_{1}}$, where $x'_{k+1} = x_{k+1}c_{k}$.

\medskip 
CASE 2: $l(g)=2k$. 

Let \[g= x_1x_2 \ldots {x_k}{x_{k+1}} \ldots x_{2k-1}x_{2k}.\]
We know $g = {\bar{g}}$
\[x_1x_2 \ldots {x_k}{x_{k+1}} \ldots x_{2k-1}x_{2k} = x_{2k}x_{2k-1} \ldots x_2x_1\]
\[\Rightarrow x_1x_2 \ldots x_{2k-1}x_{2k}{x_1}^{-1}{x_2}^{-1} \ldots {x_{2k-1}}^{-1}{x_{2k}}^{-1} = 1 \]

We know the expression on the left side is reducible. 

So, $x_{2k}{x_1}^{-1} = c_1$; for $c_1 \in C$ 
$\Rightarrow x_{2k} = {c_1}{x_1}$.

\[\Rightarrow x_1x_2 \ldots x_{2k-1}{c_1}{x_2}^{-1} \ldots {x_{2k-1}}^{-1}{x_{2k}}^{-1} = 1 \]

Further $x_{2k-1}{c_1}{x_2}^{-1} = c_2$; for $c_2 \in C$ 

$\Rightarrow x_{2k-1} = c_2{x_2}c_{1}^{-1}$.

\[\Rightarrow x_1x_2 \ldots x_{2k-2}{c_2}{x_3}^{-1} \ldots {x_{2k-1}}^{-1}{x_{2k}}^{-1} = 1 \]

In general we get $x_{2k-i} = c_{i+1}{x_{i+1}}c_{i}^{-1}$ for $1 \leq i \leq k-1$. In particular, for $i=k$, $x_{k} = c_{k+1}{x_{k+1}}c_{k}^{-1}$. This is a contradiction as the consecutive syllables lie in different factors in a reduced word.

Thus, for $g \in G$ with $l(g) = 2k$, $g$ cannot be a group-palindrome.
\end{proof}

\begin{lem}\label{lempal}
Let $g\in G$ be a product of $k$ group-palindromes, say $g= p_1p_2 \ldots p_k$. Then, $\Delta(g) \leq 12k- 9$.
\end{lem}

\begin{proof}
 Let $p$ be a group-palindrome in $G$ of non-zero length.

Then $p$ can be expressed as $p=hu\bar{h}$, where $\bar{h}$ is $h$ written in reverse and $u$ is of the form $x_i'$ from \lemref{lem2}. Then, for every $k$, $d_k(h)=d_k(\bar{h})$.

If $u= a$, we have
\[p_k(p)=_2 2p_k(h),\]
\[m_k(p)=_1 2m_k(h).\]
Then we get
\[d_k(p)=_3 2d_k(h).\]

If $u= a^{-1}$, as above, we get $d_k(p)=_3 2d_k(h)$.

If $u \neq a, a^{-1}$, we get $d_k(p)=_2 2d_k(h)$.

So, in general we have, 
\[d_k(p)=_3 2d_k(h).\]
Thus, \[r_k(p)=_3 0.\]
and $\Delta(p) \leq 3$.         

Thus, if $g\in G$ is a product of $k$ group-palindromes, say $g= p_1p_2 \ldots p_k$, then 
\begin{equation}\label{3}
\Delta(g)= \Delta(p_1p_2 \ldots p_k) \leq \Delta(p_1)+\Delta(p_2)+ \ldots +\Delta(p_k)+9(k-1) \leq 12k-9.
\end{equation}
This completes the proof.
\end{proof}

\subsubsection{Proof of \lemref{fl2}}
Now we prove that $\Delta$ is not bounded from above. For that purpose, we produce the following sequence $\{g_i\}$ for which $\Delta(g_i)$ is increasing.

\vspace{2mm}

Let $b \in B$ but not in $C$.

\medskip {Let } $g_1 = baba^{-1}ba$. Then, $p_1(g_1)= 0, p_2(g_1)= 1$ and $p_k(g_1)= 0$ for all other $k$; $m_k(g_1)= 0$ for all $k$; $d_2(g_1)= 1$ and $d_k(g_1)= 0$ for all other $k$. So, $\Delta(g_1) =1$.

\medskip Let $g_2= baba^{-1}baba^{-1}ba^{-1}ba$. Then 
$p_1(g_2)= 0, p_2(g_2)= p_3(g_2)= 1$ and $p_k(g_2)= 0$ for all other $k$, and, 
$m_1(g_2)= m_2(g_2)= 1$ and $m_k(g_2) =0$ for all other $k$. So, $\Delta(g_2) =2$.

\medskip Let $ g_3= baba^{-1}baba^{-1}ba^{-1}baba^{-1}ba^{-1}ba^{-1}ba$. Then 
$p_1(g_3)= 0, p_2(g_3)= p_3(g_3)= p_4(g_3)= 1$ and $p_k(g_3)= 0$ for all other $k$; 
$m_1(g_3)= 3, m_2(g_3)= 2$ and $m_k(g_3) =0$ for all other $k$. So, $\Delta(g_3) =4$.

\vspace{2mm}

In general, for 
\[g_n= baba^{-1}ba(ba^{-1})^2 \ldots \ldots ba(ba^{-1})^{n-1}ba(ba^{-1})^nba,\]

$p_1(g_n)= 0, ~ p_k(g_n)= 1$ for $1< k< n+1$;  $m_1(g_n)= \frac{n(n-1)}{2}, ~m_2(g_n)= n-1$ and for $k \neq 1, 2$, $m_k(g_n) =0$.  Thus we have
\[\Delta(g_n)=r_1 + r_2 + (n-1),  \]
where $r_1$ is the remainder of $\frac{n(n-1)}{2}$ divided by $2$ and $r_2$ is that of $n$ divided by $2$. So,
 
\[\Delta(g_n) \geq n-1. \]

Then, by  \eqref{3}, we get that the palindromic width of $G$ is infinite. This proves \lemref{fl2}.

\subsection{Case 2} For a non-trivial $a \in A\cup B$ such that  $CaC= Ca^{-1}C$.

\begin{lem}\label{fl3} 
Let $G=A\ast_{C} B$ be the free product of two groups $A$ and $B$ with amalgamated subgroup $C$ . Let $|A:C| \geq 3$, $|B:C| \geq 2$ and there exists an element $a \in A\cup B$ for which $CaC = Ca^{-1}C$. Then $pw(G, \{A \cup B\})$ is infinite.
\end{lem}

This lemma follows using similar methods as in \lemref{fl2}. Since the arguments will only be a slight modification of the ones used in \lemref{fl2}, we shall skip the details here. Following \cite{dob1}, we can define a quasi-homomorphism $\Delta$, same as \eqref{eq1}, and it follows that: 
\begin{lem}\label{pal1}
For any elements $g,h \in G$, $\Delta(gh) \leq \Delta(g)+ \Delta(h)+9$.
\end{lem}

For a proof of the above lemma, see  \cite{dob1, dob2}.

\medskip If $g\in G$ is a product of $k$ group-palindromes, say $g= p_1p_2 \ldots p_k$, then using  \lemref{pal1}, a version of   \lemref{lempal}  holds,  and we have, 
\begin{equation}\label{2}
\Delta(g)= \Delta(p_1p_2 \ldots p_k) \leq \Delta(p_1)+\Delta(p_2)+ \cdots +\Delta(p_k)+9(k-1) \leq 12k-9.
\end{equation}

And finally, for the same sequence used in \secref{ca1}, using \eqref{2}, we get that the palindromic width of $G$ is infinite. 

\subsection{Proof of \thmref{free}}
 The result follows by combining \lemref{fl2} and \lemref{fl3}. 

\subsubsection{The index two case}
So far we have shown that the palindromic width of $G= A{\ast}_C B$, when $|A: C| \geq 3$, $|B: C| \geq 2$, is infinite. Let's now consider the case when $|A: C| \leq 2$, $|B: C| \leq 2$.

\begin{prop}
Let $G= A{\ast}_C B$ be the free product of two groups $A$ and $B$ with amalgamated subgroup $C$ and $|A: C| \leq 2$, $|B: C| \leq 2$. If $S$ and $T$ are the generating sets of $A$ and $B$ respectively. If $pw(C, \{S, T\})$ is finite, then $pw(G, \{A \cup B\})$ is finite. 
\end{prop}

\begin{proof}
We only need to consider the case of $|A: C|= 2$, $|B: C|= 2$. Then $C$ is a normal subgroup of both $A$ and $B$.

Let $T_{A}$ and $T_{B}$ be the sets of right coset represntatives of $C$ in $A$ and $C$ in $B$. Here, $T_{A} \cong \Z_{2}$ and $T_{B} \cong \Z_{2}$.

Every element $g \in G$ can be expressed uniquely as a $C$-normal form. A $C$-normal form of $g$ is a sequence $(x_0,x_1,...,x_n)$, where $g = x_0x_1...x_n$ with $x_0 \in C$, $x_i \in T_{A}\setminus \{1\} \sqcup T_{B}\setminus \{1\}$ and consecutive $x_i$, $x_{i+1}$ lie in distinct sets. Existence and uniqueness of such a $C$-normal form follows from \cite[Theorem 11.3]{ob}.

So, for $g = x_0x_1...x_n$, where $(x_0,x_1,...,x_n)$ is the $C$-normal form of $g$, clearly, $x_1x_2...x_n \in T_{A}{\ast}T_{B} \cong \Z_2{\ast}\Z_2$.

This implies that $pw(x_1x_2...x_n) \leq 2$. Therefore, 

$pw(g) \leq pw(x_0) + pw(x_1x_2...x_n) \leq 3$.
\end{proof}





\end{document}